\documentclass{amsart}
\usepackage{setspace, amssymb, amsmath, amsfonts}
\usepackage{proofread}
\usepackage{amsaddr}

\theoremstyle{plain}
\newtheorem{thm}{Theorem}[section]
\newtheorem{theorem}[thm]{Theorem}
\newtheorem{cor}[thm]{Corollary}

\newtheorem{lemma}[thm]{Lemma}

\theoremstyle{definition}
\newtheorem{defn}[thm]{Definition}
\newtheorem{definition}[thm]{Definition}

\theoremstyle{remark}
\newtheorem{remark}[thm]{Remark}

\newcommand{\myabs}[1]{\vert#1\vert}

\newcommand{\ld}{\lambda}

\newcommand{\R}{\mathbb{R}}

\newcommand{\wtu}{\widetilde{U}}
\newcommand{\wtn}{\widetilde{N}}

\newcommand{\orb}{\mathcal O}
\newcommand{\cc}{(\widetilde{U}, G_U, \pi_U)}

\DeclareMathOperator{\iso}{Iso}

\DeclareMathOperator{\fix}{Fix}

\DeclareMathOperator{\vol}{vol}

\begin{document}

\title{You can hear the local orientability of an orbifold}

\author[S.Richardson]{Sean Richardson}
\address{Sean Richardson \\ Lewis and Clark College, Department of
Mathematical Sciences, 0615 SW Palatine Hill Road, MSC 110, Portland, OR 97219}
\email{srichardson@lclark.edu}
\author[E. Stanhope]{Elizabeth Stanhope}
\address{Elizabeth Stanhope \\ Lewis and Clark College, Department of
Mathematical Sciences, 0615 SW Palatine Hill Road, MSC 110, Portland, OR 97219}
\email{stanhope@lclark.edu}

\maketitle

\begin{abstract}
A Riemannian orbifold is a mildly singular generalization of a Riemannian manifold which is locally modeled on the quotient of a connected, open manifold under a finite group of isometries. If all of the isometries used to define the local structures of an entire orbifold are orientation preserving, we call the orbifold \emph{locally orientable}. We use heat invariants to show that a Riemannian orbifold which is locally orientable cannot be Laplace isospectral to a Riemannian orbifold which is not locally orientable.  As a corollary we observe that a Riemannian orbifold that is not locally orientable cannot be Laplace isospectral to a Riemannian manifold.
\end{abstract}

\medskip

\noindent 
\begin{center}
\begin{small}
\end{small}
\end{center}
\bigskip

\section{Introduction}

A Riemannian orbifold is a mildly singular generalization of a Riemannian manifold first introduced by I. Satake \cite{Satake56} in 1956 and later popularized by W. Thurston \cite{Th}.  The study of the spectral geometry of orbifolds was initiated by Y-J. Chiang \cite{MR1089240} who established the existence of the Laplace spectrum and heat kernel of a compact Riemannian orbifold. Results from the spectral geometry of manifolds have been extended to orbifolds, such as C. Farsi's \cite{MR1821378} proof of the Weyl law for orbifolds showing that the spectrum determines the dimension and volume of an orbifold.    E. Dryden, C. Gordon, S. Greenwald, and D. Webb \cite{dggw,MR3619744} established the asymptotic expansion of the heat trace for orbifolds, noting that the singular structure of an orbifold contributes additional terms to those familiar from the manifold setting.  The spectrum does not determine an orbifold's  singular structure however,  for example J.P. Rossetti, D. Schueth, M. Weilandt \cite{MR2447904} have shown that isospectral orbifolds can have maximal isotropy of different orders, and N. Shams, D. Webb and second author \cite{MR2263484} showed that isospectral orbifolds can have topologically distinct singular sets.

The local structure of an $n$-dimensional orbifold is that of the quotient of an open set in $\R^n$ by a finite group of diffeomorphisms.  More precisely, an orbifold coordinate chart over a neighborhood $U$ in an orbifold is a triple $\cc$ for which $\wtu$ is a connected open subset of $\R^n$,  $G_U$ is a finite group of diffeomorphisms acting on $\wtu$, and $\wtu/G_U$ is homeomorphic to $U$. If the local group $G_U$ of an orifold coordinate chart acts only by orientation-preserving transformations, we say that this chart is \emph{orientable}.  An orbifold is called \emph{locally orientable} if all of its coordinate charts are orientable.  Note that being locally orientable does not imply that an orbifold is orientable in the standard sense.  For example a Klein bottle is an orbifold with trivial, hence orientable, orbifold charts but is globally non-orientable. In this note we use heat trace methods to show that \emph{one can hear the local orientability of an orbifold}. That is, an orbifold that possesses at least one coordinate chart that is not orientable cannot be isospectral to an orbifold with all orientable coordinate charts. 

It is not known whether or not a Riemannian orbifold with nonempty singular set can be isospectral to a manifold.  Our result implies that an orientation reversing element in a coordinate chart is an obstruction to isospectrality to a manifold. This observation is equivalent to that made by E. Dryden, C. Gordon, S. Greenwald, and D. Webb in \cite[Theorem 5.1]{dggw}. The question of detecting the orientability, in the standard sense, of a manifold or orbifold from its Laplace spectrum is still unresolved in the closed setting. However P. B\'erard and D. Webb \cite{MR1322332} constructed a pair of isospectral flat surfaces with boundary of which one is orientable while the other is not.

Henceforth we assume all orbifolds are closed, compact and connected, unless otherwise stated.

\medskip

\noindent\emph{Acknowledgements.} This work was supported in part by the John S. Rogers Science Research Program at Lewis \& Clark College.  The second author also thanks Bucknell University for its hospitality during the completion of the manuscript. We also thank the reviewer for their helpful suggestions.

\section{Riemannian orbifolds and their Laplace spectra}

In this section we follow \cite{gordon12} by C. Gordon,  and \cite{dggw} by E. Dryden, C. Gordon, S. Greenwald, and D. Webb as we recall the definition and basic properties of a Riemannian orbifold, and the asymptotic expansion of the heat trace of an orbifold, respectively. 

\begin{definition} \label{defn:ofld}
Let $\orb$ be a second countable Hausdorff space, and let $U$ be a connected open subset of $\orb$.   
 \begin{enumerate}
\item[a.] An $n$-dimensional \emph{orbifold coordinate chart} over $U$ is a
    triple $\cc$ for which: $\wtu$ is a connected open subset of $\R^n$,
   $G_U$ is a finite group acting effectively on
    $\wtu$ by diffeomorphisms, and the mapping $\pi_U$
    from $\wtu$ onto $U$ induces a homeomorphism from the orbit space
    $\wtu/G_U$ onto $U$.
\item [b.] An \emph{orbifold atlas} is a collection of compatible orbifold
    charts $\cc$ such that the images $\pi_U({\wtu})$ cover $\orb$.  An \emph{orbifold} is a second countable Hausdorff space together with an orbifold atlas.  
\item[c.] Suppose $p\in U\subset \orb$ and $\cc$ is an orbifold chart over
    $U$.  The \emph{isotropy type} of $p$ is the isomorphism class of the
isotropy group of a lift $\tilde p\in\pi_U^{-1}(p)$ of $p$ under the action of
   $G_U$.  The isotropy type of $p$ is independent of the choice of lift $\tilde p$ as well as the choice of orbifold chart. Note that the isotropy type of $p$ can be canonically identified with a conjugacy class of subgroups of $O(n)$.  See \cite[Section 1.2]{gordon12} for more about this.
\item[d.] Points in $\orb$ with nontrivial isotropy are called \emph{singular points}.  Points that are not singular are called \emph{regular points}. 
\item[e.] A Riemannian structure on an orbifold is defined by giving the local cover $\wtu$ of each orbifold chart $\cc$ a $G_U$-invariant Riemannian metric so that the maps involved in the compatibility condition are isometries.  An orbifold with a Riemannian structure will be called a \emph{Riemannian orbifold}.  
\end{enumerate}
\end{definition}

An orbifold $\orb$ possesses a stratification given by its singular structure.  In particular, define an equivalence relation on $\orb$ by $p$ is \emph{isotropy equivalent} to $q$ if and only if both points have the same isotropy type. The connected components of isotropy equivalent sets of points, called \emph{$\orb$-strata}, form the leaves of the stratification. From \cite[Theorem 1.24]{gordon12} and \cite[Proposition 2.13]{dggw} we have the following properties of this stratification.  Note that a \emph{smooth stratification} of an orbifold or manifold is a locally finite partition of that orbifold or manifold into locally closed submanifolds.  

\begin{theorem}\label{stratification} Let $\orb$ be an orbifold (not necessarily compact nor connected) and $\cc$ an orbifold coordinate chart in $\orb$, then
\begin{itemize}
\item[a.]  The $\orb$-strata form a smooth stratification of $\orb$.  
\item[b.] The closure of an $\orb$-stratum $N$ is made up of the union of $N$ with a collection of lower-dimensional strata.  
\item[c.] If $\orb$ is compact, the stratification of $\orb$ is finite.
\item[d.] If $\orb$ is connected, then the set of all regular points of $\orb$ form a single stratum which is open in $\orb$ and has full dimension.
\item[e.] The action of $G_U$ on $\wtu$ gives smooth stratifications of both $\wtu$ and $U$. Strata in $\wtu$ are connected components of isotropy equivalent sets of points. Strata in $U$ are connected components of the intersection of the $\orb$-strata with $U$.  
\item[f.] Any two points in the same stratum of $\wtu$ (as defined in (e)) have the same isotropy subgroups in $G_U$. 
\item[g.] For $H$ a subgroup of $G_U$, each connected component $W$ of the fixed point set of $H$ in $\wtu$ is a closed submanifold of $\wtu$.  If a stratum arising from the $G_U$-action on $\wtu$ intersects $W$ nontrivially, that stratum must lie entirely within $W$. Thus the stratification of $\wtu$ restricts to a stratification of $W$.
\end{itemize}
\end{theorem}

\begin{remark}  In Theorem~\ref{stratification}(e) above, the strata in $\wtu$ are called \emph{$\wtu$-strata} and the strata in $U$ are called \emph{$U$-strata}.
\end{remark}

The tools of spectral geometry transfer to the setting of Riemannian
orbifolds using the local structure of these spaces. For example, given
$f\in C^\infty(\orb)$, $p\in \orb$, and $\cc$ a coordinate chart about $p$,
we compute $\Delta f(p)$ by taking the Laplacian of
$\pi_U^*(f)$ at $\tilde{p} \in\pi_U^{-1}(p)$.  As in the manifold setting, the eigenvalue spectrum of the Laplace operator of a Riemannian orbifold is a sequence
$$
0 = \ld_0 \le \ld_1 \le \ld_2 \le\dots \uparrow +\infty
$$
where each eigenvalue has finite multiplicity.  We say that two orbifolds are \emph{isospectral} if their Laplace spectra agree.

\section{Heat trace asymptotics for Riemannian orbifolds}

As in the manifold setting, an important tool in studying the spectral properties of a Riemannian orbifold $\orb$ is the heat kernel of $\orb$ given by
\[K(t,x,y)=\sum_{j=0}^\infty e^{-\lambda_j t}\varphi_j(x)\varphi_j(y)\]
where $K:(0,\infty) \times \orb \times \orb \rightarrow \R$, and $\{\varphi_j\}_{j=1}^\infty$ forms an orthornormal basis of eigenfunctions of $L^2(\orb)$.  The heat trace of $\orb$ is the following function, obtained by integrating $K(t,x,x)$ over $\orb$,
\[Z(t)= \sum_{j=0}^\infty e^{-\lambda_j t}.\]
The asymptotic behavior of $Z(t)$ as $t \rightarrow 0^+$ yields invariants called the \emph{heat invariants}, obtained by H. Donnelly \cite{MR0420743} for good orbifolds and later by E. Dryden, C. Gordon, S. Greenwald, and D. Webb \cite{dggw}  for general orbifolds.  Essential to this note is the observation that isospectral orbifolds have identical heat invariants.

To state the asymptotics of the heat trace of a Riemannian orbifold precisely we will
need the following terms from~\cite{dggw}.

\begin{defn} Let $\orb$ be an orbifold
\label{def:formulas}
\begin{itemize}
\item[a.] For $a_k$ the usual heat invariants from the manifold setting, let $$I_0=(4\pi t)^{-\dim(\orb)/2} \sum_{k=0}^\infty a_k t^k.$$ 
\item[b.] Let  $(\widetilde{U}, G_U, \pi_U)$ be an orbifold coordinate chart in $\orb$ and $\widetilde{N}$ a $\wtu$-stratum in $\wtu$. By Theorem~\ref{stratification}, all points in $\widetilde{N}$ have the same isotropy group.  This group will be denoted $\iso(\widetilde{N})$.  Define $\iso^{\max}(\widetilde{N})$ as the set of all $\gamma \in \iso(\widetilde{N})$ for which $\widetilde{N}$ is open in $\fix(\gamma)$, where $\fix(\gamma)$ denotes the set of points in $\wtu$ fixed by $\gamma$.  When $\iso^{\max}(\widetilde{N})$ is non-empty, $\widetilde{N}$ is called a \emph{primary} singular stratrum of $\widetilde{U}$.
\item[c.] Let $N$ be an $\orb$-stratum and $x\in N$.  Take $(\widetilde{U}, G_U, \pi_U)$ be an orbifold coordinate chart about $x$, $\tilde x \in \pi_U^{-1}(x)$, and let $\widetilde{N}$ be the $\widetilde{U}$-stratum through $\tilde x$.  Define
 \[b_k(N,x) = \sum_{\gamma \in \iso^{\max}(\widetilde{N})} b_k(\gamma,\tilde{x}).\]
 The function $b_k(\gamma,\tilde{x})$ is defined in \cite[Section 4.2]{dggw}.
\item[d.] For an $\orb$-stratum $N$, $$I_N=(4\pi t)^{-\dim(N)/2}\sum_{k=0}^\infty t^k \int_N b_k(N,x) d\vol_N(x).$$
\end{itemize}
\end{defn}

With this notation in place, we recall the asymptotic behavior of the heat trace of a Riemannian orbifold as $t\rightarrow 0^+$.

\begin{thm}\cite[Theorem 4.8]{dggw} \label{hta} Let $\orb$ be a Riemannian orbifold and let $0 = \ld_0 \le \lambda_1 \le \lambda_2 \le \dots $ be the spectrum of the associated Laplacian acting on smooth functions on $\orb$. The heat trace $\sum_{j=1}^{\infty}e^{-\lambda_{j} t}$ of $\orb$ is asymptotic as $t \rightarrow 0^+$ to
\begin{align}\label{expansion}
I_0+\sum_{N \in S(\mathcal{O})}\frac{I_N}{\myabs{\iso(N)}}
\end{align}
where $S(\orb)$ is the set of all singular $\orb$-strata and where
$\myabs{\iso(N)}$ is the order of the isotropy at each $p \in N$.
Notice this asymptotic expansion is of the form 
\[(4\pi t)^{-\dim{\orb}/2} \sum_{j=0}^\infty c_j t^{\tfrac{j}{2}}\]
for some constants $c_j$.
\end{thm}

\section{Main result}

We begin by defining a locally orientable orbifold, as discussed in the introduction.  Then, following a series of lemmas, we use orbifold heat invariants to show that an orbifold that is locally orientable cannot be isospectral to one that is not. 

\begin{definition} Let $\orb$ be an orbifold.
\begin{itemize}
\item[a.] An orbifold coordinate chart $\cc$ in $\orb$ is said to be \emph{orientable} if the group $G_U$ consists of orientation-preserving transformations of $\wtu$.
\item[b.] If all coordinate charts of $\mathcal{O}$ are orientable then we say that $\orb$ is \emph{locally orientable}.
\end{itemize}    
\end{definition}

\begin{lemma}\label{lem:b_0} Let $\orb$ be a Riemannian orbifold. Let $N$ be an $\orb$-stratum and $x\in N$. For a coordinate chart $(\widetilde{U}, G_U, \pi_U)$ about $x$ let $\widetilde{N}$ be the $\wtu$-stratum of a point $\tilde{x}\in \varphi_U^{-1}(x)$.  If $\iso^{\max}(\widetilde N)$ is non-empty then $b_0(N,x) > 0$.
\end{lemma}

\begin{proof} From \cite[p.16]{dggw} we have $$b_0(N,x) = \sum_{\gamma \in \iso^{\max}(\widetilde N)} \myabs{\det(B_\gamma(\tilde{x})) }$$
where $B_\gamma(\tilde{x})$ is a non-singular matrix.  Because $\iso^{\max}(\widetilde N)$ is non-empty, we see that $b_0(N,x)$ is the sum of a list of positive numbers.
\end{proof}

\begin{lemma}\label{lem:dim_of_fix} Let $\orb$ be a Riemannian orbifold. Suppose $(\widetilde{U}, G_U, \pi_U)$ is a coordinate chart in the orbifold $\orb$ and let $\widetilde{N}$ be a $\widetilde{U}$-stratum.  Let $\gamma \in \iso(\wtn)$. Then $\gamma \in \iso^{\max}(\wtn)$ if and only if $\dim(\fix(\gamma)) = \dim(\widetilde{N})$.
\end{lemma}

\begin{proof}
If $\gamma \in \iso^{\max}(\wtn)$ then $\wtn$ is open in the submanifold $\fix(\gamma)$, implying $\dim(\widetilde{N}) = \dim(\fix(\gamma))$. For the reverse direction recall that Theorem~\ref{stratification}(g) states that each connected component of $\fix(\gamma)$ (more precisely the fixed point set of the cyclic group generated by $\gamma$, which equals $\fix(\gamma)$)  is stratified by a set of $\wtu$-strata, one of which is $\widetilde{N}$.  By \cite[Remark 2.9(i)]{dggw} maximum dimensional strata are open.  We see $\widetilde{N}$ is open in $\fix(\gamma)$, thus $\gamma \in \iso^{\max}(\wtn)$.
\end{proof}

 The paper by H. Donnelly \cite{MR0433513} was helpful in the development of the following lemma.

\begin{lemma}\label{lem:dim-of-ori-rev} Let $\orb$ be a Riemannian orbifold and $\cc$ a coordinate chart in $\orb$. Suppose $\gamma \in G_U$ and $\dim(\fix(\gamma))<\dim(\orb)$. Then, $\dim(\fix(\gamma))$ is of opposite parity to $\dim(\mathcal{O})$ if and only if $\gamma$ is orientation reversing.   
\end{lemma}  

\begin{proof}  For simplicity write $\dim(\orb)=n$ and $\dim(\fix(\gamma))=d$.  Let $W$ be a connected component of $\fix(\gamma)$. Because $\wtu$ is connected it suffices to show that at some point $p\in W$ the differential of $\gamma$, denoted $\gamma_{*p}$, is orientation reversing exactly when $n$ and $d$ have opposite parity.  For any $p\in W$ we have that $\gamma_{*p}$ acts trivially on $T_pW$ and that
\[(T_pW)^\perp =  (T_pW)^{\perp}(-1) \oplus  (T_pW)^{\perp}(\theta_1) \oplus \dots \oplus  (T_pW)^{\perp}(\theta_\ell)\]
where each $\theta_i \in (0,\pi)$,  $\gamma_{*p}$ acts on $(T_pW)^{\perp}(-1)$ by multiplication by $-1$, and each $(T_pW)^{\perp}(\theta_i)$ has even dimension and is acted upon by $\gamma_{*p}$ by a direct sum of rotations by the angle $\theta_i$.  Now $n$ and $d$ have opposite parity exactly when $\dim((T_pW)^\perp)=n-d$ is odd. This can only occur if $(T_pW)^{\perp}(-1)$ is odd dimensional, in particular when $\gamma$ is orientation reversing.
\end{proof}

\begin{definition} Let $\orb$ be a Riemannian orbifold. If the dimension of an $\orb$-stratum $N$ has opposite parity to the dimension of $\orb$, we call $N$ an \emph{opposite parity stratum} of $\orb$. For convenience the phrase ``opposite parity stratum" will be abbreviated to ``OP-stratum."
\end{definition}

\begin{lemma}\label{orinop} Let $\orb$ be a Riemannian orbifold. Then $\orb$ is locally orientable if and only if $\orb$ has no primary OP-strata.
\end{lemma}

\begin{proof}
Suppose $\orb$ is not locally orientable. Then there is a coordinate chart $\cc$ in $\orb$ with an orientation reversing element $\gamma \in G_U$. Lemma~\ref{lem:dim-of-ori-rev} implies $\fix(\gamma)$ has dimension of opposite parity to the dimension of $\orb$. Suppose $W$ is a connected component of  $\fix(\gamma)$. Theorem~\ref{stratification}(g) implies $W$ is stratified by a finite set of $\wtu$-strata $N_1, N_2, \dots, N_r$. So for at least one $i_0 \in \{1, 2, \dots, r\}$, the stratum $N_{i_0}$ must have the same dimension as $\fix(\gamma)$. Lemma~\ref{lem:dim_of_fix} implies $\gamma \in \iso^{\max}(N_{i_0})$. Thus $N_{i_0}$ is the required primary OP stratum.

Suppose $\orb$ has a primary OP stratum $N$ and take $\gamma \in \iso^{\max}(N)$. Lemma~\ref{lem:dim_of_fix} implies $\dim(\fix(\gamma))= \dim(N)$. Thus $\fix(\gamma)$ has dimension of opposite parity to the dimension of $\orb$. By Lemma~\ref{lem:dim-of-ori-rev} we conclude $\gamma$ is orientation reversing.
\end{proof}

\begin{theorem}  A locally orientable orbifold cannot be isospectral to an orbifold that is not locally orientable.
\end{theorem}

\begin{proof} Consider a locally orientable orbifold  $\orb_{ori}$ and a non-locally orientable orbifold $\orb_{non}$.  If these orbifolds have different dimensions then the Weyl law for orbifolds \cite{MR1821378} implies they cannot be isosopectral.  
So we can assume $\dim(\orb_{ori})= \dim(\orb_{non}):=n$.  In the case that $n$ is odd, Lemma~\ref{orinop} implies $\orb_{ori}$ will have no even dimensional primary strata while
$\orb_{non}$ will have at least one even dimensional primary stratum.  

We argue that in this situation all integer power terms in the heat expansion of $\orb_{ori}$ vanish.  To see this first observe that because $n$ is odd $I_0$, the first term of the heat expansion as stated in Theorem \ref{hta}, consists of only half integer terms so any integer power terms would have to arise in the second term of this expansion.  We use Definition~\ref{def:formulas} to give the following detailed expression of the second term,
\begin{equation*}
    \sum_{N \in S(\mathcal{O})}\frac{{(4\pi
    t)}^{-\dim(N)/2}}{\myabs{\iso(N)}}\sum_{k=0}^{\infty}t^k\int_{N}
    \sum_{\gamma \in \iso^{\max}(\widetilde{N})}b_k(\gamma,x) dvol_N \ .
\end{equation*}
Notice that no integer terms can arise in this second term as $\orb_{ori}$ lacks even dimensional primary strata.

We now show that at least one integer power coefficient in the expansion of
$\orb_{non}$ is nonzero. Let $d$ denote the maximum dimension of all the strata in the set of even dimensional primary strata in $\orb_{non}$. Note that only these strata of maximal
dimension will contribute to the $-d/2$ term in the heat expansion, which
occurs in the $k=0$ iteration in the sum. Furthermore, by Lemma~\ref{lem:b_0} the
$b_0$ term for each contributing strata is strictly positive. Thus the
integer $-d/2$ term is the sum of strictly positive terms and so must be nonzero. 

Since $\orb_{orb}$ and $\orb_{non}$ differ in at least one term in the heat
expansion they cannot be isosopectral.  When $n$ is even the proof proceeds similarly, reversing the roles of integer and half-integer terms.

\end{proof}

We end with a corollary that is equivalent to \cite[Theorem 5.1]{dggw}.

\begin{cor} Let $\orb$ be an orbifold.  If any local chart $\cc$ on $\orb$ possesses an orientation reversing local group element $\gamma \in G_U$, then $\orb$ cannot be isospectral to a manifold.
\end{cor}


\bibliographystyle{plain}
\bibliography{localoribib}

\end{document}